%% file: main.tex
\documentclass{amsart}

\usepackage{lmodern}
\usepackage[T1]{fontenc}
\usepackage{amsmath}
\usepackage{amssymb}
\usepackage{amsthm}
\usepackage{url}
\usepackage{microtype}
\usepackage{needspace}
\usepackage{array}
\usepackage{float}
\usepackage[hidelinks]{hyperref}

\input{generated_curves}

\input{generated_constants}

\theoremstyle{plain}
\newtheorem{theorem}{Theorem}
\newtheorem{lemma}{Lemma}
\theoremstyle{definition}
\newtheorem{question}{Question}

\numberwithin{equation}{section}

\title[On $a^4+b^4+c^4+d^4=\bigl(a+27(b+c+d)\bigr)^4$]
{On $a^4+b^4+c^4+d^4=\bigl(a+27(b+c+d)\bigr)^4$\\ and ten other new infinite families}

\author{Matej Veselovac}
\address{(M. Veselovac) Independent researcher}
\email{vepirson@gmail.com}

\author{Tito Piezas III}
\address{(T. Piezas III) Independent researcher}
\email{tpiezas@gmail.com}


\begin{document}

\begin{abstract}
In 2008, Jacobi and Madden proved that $a^4+b^4+c^4+d^4=e^4$ has infinitely
many primitive solutions with the linear relation $e=a+b+c+d$.
We extend their result to relations with rational cubes $e=a+k^3(b+c+d)$ and prove infinitude
for $k=3$ and $10$ other $k\ne 1$. We compile $\seedCount$ starting solutions.
In the spirit of Jacobi and Madden's original paper, everything here is simple algebra except one theorem of Mazur.
\end{abstract}

\maketitle

\section{Introduction}
\label{sec:intro}

In 1772, Euler~\cite{Eu, Di} conjectured that the sum of three fourth powers is never a fourth power and that the sum of four fourth powers can equal a fourth power,
\begin{equation}
\label{eq:main}
a^4+b^4+c^4+d^4=e^4.
\end{equation}
In 1988, Elkies~\cite{Elk} found infinitely many primitive solutions for the case $d=0$. From now on, let $abcde\ne 0$.
Solutions of \eqref{eq:main} were found by Norrie in 1911~\cite{Nor} and by others, before Rose and Brudno in 1973~\cite{RB} tabulated $82$ primitive solutions.

In 2008, Jacobi and Madden~\cite{JM} proved that infinitely many primitive solutions of \eqref{eq:main} satisfy the linear relation
\begin{equation}
\label{eq:jm}
e=a+b+c+d.
\end{equation}
They used Brudno's seed $955^4+1770^4+(-2634)^4+5400^4=5491^4$~\cite{Bru}. Later, more seeds for \eqref{eq:jm} were found by Tomita~\cite{To1,To2} and MacLeod~\cite{ML}. Piezas~\cite{Pz} compiled a list of $63$ seeds with $e\le 10^{18}$ which Veselovac~\cite{Vea} extended to $126$.

In 2026, Go~\cite{Go}
computed all $45{,}866$ primitive solutions of \eqref{eq:main} with
$e\le10^{7}$.
While mining Go's table for new linear relations in
$a,b,c,d,e$, we discovered~\cite{Ve} that other primitive solutions 
satisfy new linear relations for $k=\frac{n}{m}$, $\gcd(n,m)=1$,
\begin{equation}
\label{eq:slice}
e=a+k^3\,(b+c+d).
\end{equation}
For instance, the $k=3$ relation $e=a+27(b+c+d)$ is carried by the long-known
Rose--Brudno seed
$13355^4+1530^4+(-9498)^4+8010^4=14489^4$~\cite{RB}.

We list $k=3$ and $10$ other $k\ne 1$. We compile their $\seedCount$ seeds (Table~\ref{tab:seeds}) and derive $11$ infinite families (Theorem~\ref{thm:main}), each from one seed placement (Table~\ref{tab:seeddata}).
One identity fits them all (Lemma~\ref{lem:identity}, \cite{MO}) and gives explicit quartic curves $w^2=Q(r)$ in three polynomials $A$, $B$, $C$ (Lemma~\ref{lem:abc}, \cite{MSEa}), see $k=3$ (example \S\ref{sec:example}).

In other words, we derive (identity \S\ref{sec:identity}, parameterization \S\ref{sec:parameterization}, families \S\ref{sec:families}) elliptic curves (Table~\ref{tab:curves}), 
with infinitely many rational points (\S\ref{sec:points}) and primitive solutions (\S\ref{sec:solutions}). 
Everything below is simple algebra except one theorem of Mazur~\cite{Maz}.

\section{Result}
\label{sec:result}

A solution of \eqref{eq:main} is \emph{primitive} if $\gcd(a,b,c,d,e)=1$.
A \emph{seed} is a primitive solution satisfying \eqref{eq:slice} for some \emph{relation} $k$ determined by $k^3=(e-a)/(b+c+d)$.
Let $k=\frac{n}{m}$ with $\gcd(n,m)=1$, $m\ge1$, $n>0$ and $e>0$, so that
\eqref{eq:slice} reads
\begin{equation}
\label{eq:intform}
m^3e=m^3a+n^3(b+c+d).
\end{equation}

\begin{theorem}
\label{thm:main}
Each of the eleven $k$ gives infinitely many primitive solutions of \eqref{eq:main}\textup{:}
\begin{align*}
\theoremList .
\end{align*}
\end{theorem}

\input{generated_seed_table}

\begin{proof}
For each given $k$, we follow a three-step recipe to a new infinite family.
\begin{enumerate}
\item \emph{Elliptic curves}
(identity \S\ref{sec:identity}, parameterization \S\ref{sec:parameterization}, families \S\ref{sec:families}).
By Lemma~\ref{lem:identity}, any seed (Table~\ref{tab:seeds}), in any of its placements (Table~\ref{tab:seeddata}), gives a rational point $(p_0,D_0)$ of a conic $D^2=G(h,k,p)$.
Take one seed and one placement.
By Lemma~\ref{lem:abc}, the lines through this point parameterize
the conic, and substituting into the remaining square condition
gives the quartic $w^2=Q(r)$,
determined by the polynomials $A,B,C$.
The seed corresponds to a rational point $T$ on the quartic, hence on an elliptic curve (Table~\ref{tab:curves}).
\item \emph{Infinitely many rational points} (\S\ref{sec:points}).
We show that $T$ is not a torsion point.
The multiples of $T$ contain $18$ distinct points (Table~\ref{tab:cert}), giving it infinite order by Mazur~\cite{Maz}.
Hence the curve has infinitely many rational points.
\item \emph{Infinitely many primitive solutions} (\S\ref{sec:solutions}).
By substitutions \eqref{eq:abc} and \eqref{eq:subst},
rational points on the quartic give rational solutions of \eqref{eq:norm}, hence primitive solutions of \eqref{eq:main}.
Distinct points give distinct solutions, up to a sign, and only finitely many points have a zero coordinate. Hence infinitely many primitive solutions remain.
For example, $2T$ gives a new primitive solution.
\end{enumerate}
The families for distinct relations \eqref{eq:slice}, $k_1\ne k_2$, share no solution on the same coordinate $a$. If not, $b+c+d=0$, so $e=a$ and then \eqref{eq:main} gives $b=c=d=0$.
\end{proof}

In Table~\ref{tab:seeds}, we computed the $\seedsNew$ new seeds with $e>10^7$
(seed discovery \S\ref{sec:computation}),
while the $\seedsInGo$ rows with $e\le10^7$ were known~\cite{RB,Wro,Go}.

\section{Identity}
\label{sec:identity}

A \emph{placement} of a seed chooses which of $b,c,d$ is substituted as
$z$ and calls the other two $x>y$, all divided by $a$. Then $(x,y,z)$ is
a rational solution of
\begin{equation}
\label{eq:norm}
1+x^4+y^4+z^4=\bigl(1+k^3(x+y+z)\bigr)^4 .
\end{equation}
Conversely, up to sign, distinct rational solutions of \eqref{eq:norm}
with $xyz\ne0$ are distinct primitive solutions of \eqref{eq:main}
divided by $a$. Do the substitutions
\begin{equation}
\label{eq:subst}
x=\frac{p+q}{2},\qquad y=\frac{p-q}{2},\qquad z=k+(h-1)(k+p),
\end{equation}
with a free parameter $h$. The
equation becomes a biquadratic in $q$, whose solution is
\begin{equation}
\label{eq:disc}
q^2=-3p^2\pm2(k+p)\sqrt{G},
\end{equation}
\[
G=8h(k^2-kp+p^2)+12h^2(k^6-p^2)+8h^3(k^9+p)(k+p)+2h^4(k^{12}-1)(k+p)^2 .
\]
The powers $(k^6,k^9,k^{12})$ explain the large coefficients of the
curves (Table~\ref{tab:curves}). 

The fact that $G$ is only a quadratic in $p$ immediately implies the method. If there is a rational $p$ such that $G$ is a square, then $p$ can be replaced by a ratio of two quadratics in some new variable $r$. Substituting this ratio into the formula for $q^2$, we then get an equation of the form $q^2=Q(r)/C(r)^2$ with $Q$ a quartic in $r$ (Lemma~\ref{lem:abc}). From the known rational point $p$, we can derive a rational point $r$.

Piezas first presented the identity on MathOverflow~\cite{MO} and then parameterized it on Mathematics Stack Exchange~\cite{MSEa}. We state the identity and parameterize it as Lemma~\ref{lem:identity} and Lemma~\ref{lem:abc}, then illustrate them on $k=3$ (example \S\ref{sec:example}). 

For $k=1$, the parameter $h$ of \eqref{eq:subst} is the ratio $(e-a)/(e-d)$ when $z=d/a$. In their paper, Jacobi and Madden~\cite[Section~3]{JM} have curves on which such a ratio is constant, but they do not use it as a parameter.

\begin{lemma}[{\cite{MO}}]
\label{lem:identity}
Let $h,k,p,q$ be rational. With \eqref{eq:subst},
\eqref{eq:norm} holds exactly when 
\[(3p^2+q^2)^2=4\,(k+p)^2\,G.\]
\end{lemma}

\begin{proof}
Since $x=(p+q)/2$ and $y=(p-q)/2$, we have $x+y=p$ and $8(x^4+y^4)=p^4+6p^2q^2+q^4$. The substitution turns
$8\bigl[1+x^4+y^4+z^4-\bigl(1+k^3(x+y+z)\bigr)^4\bigr]$ into $F+6p^2q^2+q^4$, where we have
$F=p^4+8\bigl(1+z^4-\bigl(1+k^3(p+z)\bigr)^4\bigr)$. After completing the square, this
vanishes exactly when $(3p^2+q^2)^2=9p^4-F$, and $9p^4-F=4(k+p)^2G$ holds by expansion.
\end{proof}

For fixed $h$ and $k$, the rational solutions of \eqref{eq:norm} with
$k+p\ne0$ are the rational points $(p,D)$ of the conic $D^2=G(h,k,p)$, where
\[D=(3p^2+q^2)/(2(k+p)),\]
at which the \emph{square condition} $q^2=-3p^2+2(k+p)D$ holds for some rational $q$.

\section{Parameterization}
\label{sec:parameterization}

A solution $(x,y,z)$ of \eqref{eq:norm} with $k+x+y\ne0$ has
$p_0=x+y$, $q_0=x-y$ and
\begin{equation}
\label{eq:h}
h=\frac{x+y+z}{k+x+y},\qquad D_0=\frac{3p_0^2+q_0^2}{2(k+p_0)}.
\end{equation}
So it is a rational point $(p_0,D_0)$ of its conic. The lines through
it, of slope $t=4r/h_d^2$ with $h_d>0$ the denominator of $h$ in lowest
terms, parameterize the conic by $r$. Substituting the point $(p,D)$ of
the line into the square condition gives the
quartic $w^2=Q(r)$ of \eqref{eq:Q}. 

Jacobi and Madden~\cite{JM} take
the same approach for $k=1$.
The difference is that they turn \eqref{eq:main} with \eqref{eq:jm} into a Pythagorean triple and split it into two quadric surfaces. For $k\ne1$, no such splitting is known~\cite{MSE}. 

\begin{lemma}[{\cite{MSEa}}]
\label{lem:abc}
Let $h$, $p_0$, $q_0$ be as above, with $h\ne0$, $k+p_0\ne0$, $q_0\ne0$. There are
polynomials $A$ of degree one, $C=r^2-\gamma$ and $B$ of degree at most three,
with rational coefficients, such that for every rational $r$ with $C(r)\ne0$,
\begin{equation}
\label{eq:abc}
p=p_0-\frac{6A}{hC},\qquad q^2=q_0^2-\frac{9AB}{C^2}
\end{equation}
satisfy $(3p^2+q^2)^2=4(k+p)^2G$. Hence $(x,y,z)$ of \eqref{eq:subst}
solves \eqref{eq:norm} whenever
\begin{equation}
\label{eq:Q}
Q(r)=q_0^2C^2-9AB
\end{equation}
is a rational square. The root $r_0$ of $A$ returns the seed when $C(r_0)\ne0$.
\end{lemma}

\begin{proof}
Write $G=G_2p^2+G_1p+G_0$ in $p$, so $D_0^2=G(p_0)$ on the conic. The
line $D=D_0+t(p-p_0)$ meets the conic $D^2=G$ where
\[
t^2(p-p_0)^2+2tD_0(p-p_0)=G_2(p^2-p_0^2)+G_1(p-p_0).
\]
Cancel $p-p_0$. The second point has
\[
p-p_0=\frac{G_1+2G_2p_0-2tD_0}{t^2-G_2},\qquad D=D_0+t(p-p_0).
\]
With $t=4r/h_d^2$ the denominator is $16C/h_d^4$, with $C=r^2-\gamma$,
$\gamma=G_2h_d^4/16$. Set
\[
A=-\frac{hC(p-p_0)}{6}=\frac{h\,h_d^2D_0}{12}\,r-\frac{h\,h_d^4(G_1+2G_2p_0)}{96},
\]
linear as $D_0\ne0$, so $pC=p_0C-6A/h$, $DC=D_0C-24rA/(hh_d^2)$ have degree $\le2$.
Hence
\[
Q=q^2C^2=-3(pC)^2+2(kC+pC)\,DC
\]
is a polynomial of degree at most four, and
\[
(3p^2+q^2)^2=4(k+p)^2D^2=4(k+p)^2G.
\]
At the root $r_0$ of $A$ we have $pC=p_0C$ and $DC=D_0C$, so
$Q(r_0)=q_0^2C(r_0)^2$. Hence the
linear $A$ divides $q_0^2C^2-Q$, and \eqref{eq:abc} holds with $B$ of degree at most three,
\[
B=\frac{q_0^2C^2-Q}{9A}.
\]
If $C(r_0)\ne0$, then \eqref{eq:abc} at $r_0$
gives $p=p_0$ and $q^2=q_0^2$, the seed up to sign.
\end{proof}

Similar constructions are known in the literature. For fixed $h$, the curve in $(p,q)$ is a plane quartic with two double points, as in van Luijk~\cite{vL}, but his
double points are rational and ours conjugate.
Bremner, Choudhry and Ulas~\cite{BCU} parameterize a conic through a known point and reach a curve on diagonal quartic surfaces.

\section{The example \texorpdfstring{$k=3$}{k=3}}
\label{sec:example}

The smallest seed for $k=3$ is $13355^4+1530^4+(-9498)^4+8010^4=14489^4$
(Table~\ref{tab:seeds}). Dividing by $a=13355$ gives the rational solution
\[
(x,y,z)=\left(\frac{d}{a},\;\frac{c}{a},\;\frac{b}{a}\right)=\left(\frac{1602}{2671},\;-\frac{9498}{13355},\;\frac{306}{2671}\right)
\]
of the equation
\[
1+x^4+y^4+z^4=\bigl(1+27(x+y+z)\bigr)^4.
\]
Then $h=(x+y+z)/(k+x+y)=2/1837$. The substitutions
\[
x=\frac{p+q}{2},\qquad y=\frac{p-q}{2},\qquad z=k+(h-1)(k+p)
\]
give the initial solution
\[
(p_0,q_0)= \left(-\frac{1488}{13355},\; \frac{17508}{13355}\right),
\] from which we find the parameterization
\[ p = -\frac{1488}{13355}- \frac{6A}{h\,C}\]
\[q^2 = \left(\frac{17508}{13355}\right)^2-\frac{9AB}{C^2}\]
where $(A,B,C)$ are polynomials in $r$ (Table~\ref{tab:seeddata} and Table~\ref{tab:coef}, row $k=3$),
\[A = (48044056139 + 1242748 r)/13355,\] 
\[B = (421773010904633577 + 11850222510380 r + 495920 r^2 + 56 r^3)/13355,\]
\[C = r^2-6190029774.\]
This is valid for any $r$ with $C\ne0$, and $q$ is rational exactly when
$Q=q_0^2C^2-9AB$ is a square. The root $r_0=-48044056139/1242748$ of $A$
returns the initial solution, as the point $T=\bigl(r_0,q_0C(r_0)\bigr)$ of the quartic $w^2=Q(r)$, an elliptic curve in integral form $(Mw)^2=N(r)$ (Table~\ref{tab:curves}, row $k=3$).
The point $2T$ gives a new solution with $\digitsTwoTKThree$ digits in $e$ (primitive solutions \S\ref{sec:solutions}).

\section{Families}
\label{sec:families}

We build the quartic $w^2=Q(r)$
of \eqref{eq:Q} from one seed and one placement. The choice is arbitrary. We take the seed
with the smallest $e$ (Table~\ref{tab:seeds}) and the placement whose
$h$ of \eqref{eq:h} has the smallest height, where the height of $h=h_n/h_d$ in lowest terms
is $\max(|h_n|,|h_d|)$.

For the seed of smallest $e$ of each $k$ (Table~\ref{tab:seeds}), we list its three values $h_1,h_2,h_3$ of \eqref{eq:h}, one for each $z=b/a,c/a,d/a$, then the placement $h$ we use and its point $(p_0,q_0)$ (Table~\ref{tab:seeddata}).

\input{generated_seed_data}

Write the polynomials of Lemma~\ref{lem:abc} as
\begin{equation}
\label{eq:ABCform}
A=\frac{a_0+a_1r}{d_A},\qquad
B=\frac{b_0+b_1r+b_2r^2+b_3r^3}{d_B},\qquad
C=r^2-\frac{c_0}{d_C},
\end{equation}
with integers $a_i$, $b_i$, $c_0$ and $d_A,d_B\ge1$, $d_C\in\{1,4\}$.
The root of $A$ is $r_0=-a_0/a_1$.
The chosen seed and placement (Table~\ref{tab:seeddata}) determine the polynomials (Table~\ref{tab:coef}). 

We list the curve $w^2=Q(r)$ of each family as in \eqref{eq:Q}, in its integral form
$(Mw)^2=N(r)$, $N=M^2Q$, where
\begin{equation}
\label{eq:Nform}
N(r)=n_4r^4+n_3r^3+n_2r^2+n_1r+n_0.
\end{equation}
$N$ has integer coefficients with no common square factor (Table~\ref{tab:curves}).

For $k=2,3,5,5/3$, up to the scaling of $r$, the tables (Table~\ref{tab:seeddata}, Table~\ref{tab:coef}, Table~\ref{tab:curves}) were also derived by Piezas~\cite{MO,MSEa}. For $k=15$ he takes $h$ as $h_3$ instead of $h_2$, a different infinite family on the same seed and the same relation.

\input{generated_families}

\input{generated_curve_table}

Since the choice of seed and placement is arbitrary, up to $\familyCount$ infinite families arise the same way from the $\seedCount$ seeds and their $\familyPerSeed$ placements. For brevity, we derive one infinite family for each of the $11$ relations $k$ (proof of Theorem~\ref{thm:main}).

\section{Infinitely many rational points}
\label{sec:points}

The quartic $w^2=Q(r)$ of \eqref{eq:Q} is an elliptic curve $E$. It carries the seed's rational point $T=\bigl(r_0,q_0C(r_0)\bigr)$ and is nonsingular, as the coefficients show (Table~\ref{tab:curves}).

Take the origin $O=\bigl(r_0,-q_0C(r_0)\bigr)$.
The parabola $w=g(r)$ with $\deg g\le2$ through the points $X$, $Y$ and
$T$ meets $E$ in a fourth point, paired with $X+Y$. The pair of points $(r,\pm w)$ sums to the same point for every $r$, hence to the point $O+T=T$ at $r=r_0$~\cite[Ch.~8]{Cas}. So $jT$ and $(1-j)T$ are paired. We compute $T,2T,\dots,9T$ exactly by this rule.
Jacobi and Madden~\cite{JM} use the same rule.

The prime $\ell$ divides no denominator of these nine points, and their
residues modulo $\ell$ have distinct $r$ and $w\not\equiv0$
(Table~\ref{tab:cert}). So the $18$ multiples $jT$ and $(1-j)T$,
$j=1,\dots,9$, are distinct rational points. A torsion point has at most
$16$ distinct multiples by Mazur~\cite{Maz} (see also~\cite[p.~58]{ST}),
so $T$ has infinite order.

\input{generated_appendix}

\section{Infinitely many primitive solutions}
\label{sec:solutions}

A rational point $(r,w)$ of the quartic gives a primitive solution of
\eqref{eq:main} by substitutions \eqref{eq:abc} and \eqref{eq:subst}. 
The maps $(r,w)\mapsto(p,D,q)\mapsto(x,y,z)$ are
injective, since a line is determined by its second intersection with
the conic. The point $(r,-w)$ exchanges $x$ and $y$, so each solution
of \eqref{eq:main} comes from at most two points. A zero coordinate is
a polynomial equation in $r$, met by finitely many points. Hence
infinitely many primitive solutions remain.

\input{generated_solutions}

Whether the quartics have rational points smaller than their seeds, and hence smaller primitive solutions, is open~\cite[question~VI]{MO}.

\section{Seed discovery}
\label{sec:computation}

We computed the $\seedsNew$ new seeds with $e>10^7$ (Table~\ref{tab:seeds}) by extending
MacLeod's quartic method~\cite{ML} from $k=1$ to other relations $k\ne 1$.
For fixed $k$ and $h$, the solutions of \eqref{eq:norm} are the rational points of
one curve of genus one. The parameterization of Lemma~\ref{lem:abc} needs no seed.
Any rational point of the conic $D^2=G$ parameterizes it and gives a quartic model
$w^2=Q(r)$ of the curve.

For each relation we sweep $h$ by height. For each $h$ we find a rational point of
the conic with \texttt{qfsolve} of PARI/GP~\cite{PARI}, or discard $h$ if there is
none. The height of $r$ at a solution depends on the model. So we take several
models of the curve and minimize and reduce them with \texttt{hyperellminimalmodel}
and \texttt{hyperellred}. We search each model with \texttt{hyperellratpoints} for
rational points. Each point found gives a solution of \eqref{eq:main}. 
Using the same method, the seven $k=3$ seeds with $e>10^7$ were also found by Veselovac~\cite{MOa}.

The relations were searched with $h$ and $r$ of height up to $10^5$. The two largest
seeds, both of $k=5/3$, sit beyond this range. We found them by a $4$-descent
on the free online calculator of Magma~\cite{Mag}. 

The
full pipeline, with more new relations $k$, is the subject of a
follow-up paper. So far, Veselovac~\cite{MOb} compiled a list of $64$ seeds
for $27$ distinct relations $k\ne 1$. The solutions to \eqref{eq:main} carried by infinite families are also listed by Perrenet~\cite{Pe}.

\section{Open questions}
\label{sec:open}

\begin{question}
Does every relation $k$ carry a seed?
Any seed should give an infinite family by the same recipe (proof of Theorem~\ref{thm:main}). In other words, proving that an infinite subset of relations $k$ carries elliptic curves is an open problem.
\end{question}

\begin{question}
Does there exist a split into two quadric surfaces for some $k\ne1$~\cite{MSE}?
Jacobi and Madden~\cite{JM} split the case $k=1$ into two quadric surfaces.
\end{question}

\begin{question}
Do nonlinear relations in $a,b,c,d,e$ carry infinite families of \eqref{eq:main}? There is another set of infinite families on \eqref{eq:main} beside \eqref{eq:slice}, the relation $a=b$~\cite{MOc}, due to a 2014 identity of Bremner, Choudhry and Ulas~\cite{BCU} and a 2026 seed~\cite{Go}.
\end{question}

\section*{Note on use of AI}

Claude Fable (Anthropic) was used to data mine Go's table~\cite{Go} of
all $45{,}866$ primitive solutions of \eqref{eq:main} with $e\le10^{7}$.
The authors proposed it should focus on new linear relations
$\pm \alpha_1 a \pm \alpha_2 b \pm \alpha_3 c \pm \alpha_4 d = \pm \alpha_5 e$. 
It wrote search scripts and performed symbolic analysis. 
Eventually, it singled out the relation with $k=3$ and proposed that it carries an infinite family,
suggesting we check other values of $k$.

It was used to write a search program to our specifications (seed discovery \S\ref{sec:computation}),
and the verification program in exact arithmetic for every number in this paper.

It was used for drafts of the paper, which was finally rewritten by the authors.

\section*{Acknowledgments}

The authors thank Eugene Go for computing and publishing the
table~\cite{Go}. The smallest seeds of $k=3$ and $k=15$ had been in print since
Rose and Brudno's table of 1973~\cite{RB} and Wroblewski's list of
2006~\cite{Wro} respectively.

The authors also thank David, who asked on Mathematics Stack Exchange
for more numerical solutions of \eqref{eq:main} and for relations
between $a,b,c,d,e$ other than Jacobi--Madden's~\cite{Dav},
and everyone who took part in the answers and comments.

\end{document}

%% file: generated_curves.tex
\newcommand{\digitsTwoTKTwo}{108}

\newcommand{\dAKTwo}{86{,}329{,}625}
\newcommand{\dBKTwo}{28{,}747{,}765{,}125}
\newcommand{\dCKTwo}{1}

\newcommand{\hbKTwo}{-\dfrac{296}{909}}
\newcommand{\hcKTwo}{-\dfrac{951176}{7774639}}
\newcommand{\hdKTwo}{-\dfrac{35193512}{310840585}}
\newcommand{\hUsedKTwo}{h_1}
\newcommand{\pzeroSlashKTwo}{-\dfrac{100607668}{86329625}}
\newcommand{\qzeroSlashKTwo}{\dfrac{15452628}{86329625}}
\newcommand{\MKTwo}{86{,}329{,}625/4}

\newcommand{\jmRowRKTwo}{36 & 10 & 22 & 42 & 0 & 35 & 27 & 7 & 28}
\newcommand{\jmRowWKTwo}{\pm19 & \pm19 & \pm11 & \pm3 & \pm4 & \pm8 & \pm4 & \pm1 & \pm20}

\newcommand{\digitsTwoTKThree}{73}

\newcommand{\dAKThree}{13{,}355}
\newcommand{\dBKThree}{13{,}355}
\newcommand{\dCKThree}{1}

\newcommand{\hbKThree}{\dfrac{2}{1837}}
\newcommand{\hcKThree}{\dfrac{14}{16535}}
\newcommand{\hdKThree}{\dfrac{14}{10699}}
\newcommand{\hUsedKThree}{h_1}
\newcommand{\pzeroSlashKThree}{-\dfrac{1488}{13355}}
\newcommand{\qzeroSlashKThree}{\dfrac{17508}{13355}}
\newcommand{\MKThree}{13{,}355/3}

\newcommand{\jmRowRKThree}{10 & 15 & 6 & 21 & 26 & 29 & 25 & 0 & 19}
\newcommand{\jmRowWKThree}{\pm10 & \pm9 & \pm13 & \pm13 & \pm11 & \pm11 & \pm5 & \pm3 & \pm4}

\newcommand{\digitsTwoTKFive}{98}

\newcommand{\dAKFive}{2{,}287{,}976}
\newcommand{\dBKFive}{69{,}497{,}271}
\newcommand{\dCKFive}{4}

\newcommand{\hbKFive}{\dfrac{3483}{561884}}
\newcommand{\hcKFive}{\dfrac{81}{24245}}
\newcommand{\hdKFive}{\dfrac{3483}{118529}}
\newcommand{\hUsedKFive}{h_2}
\newcommand{\pzeroSlashKFive}{\dfrac{2352705}{571994}}
\newcommand{\qzeroSlashKFive}{\dfrac{2216775}{571994}}
\newcommand{\MKFive}{2{,}287{,}976/15}

\newcommand{\jmRowRKFive}{2 & 12 & 5 & 0 & 23 & 14 & 15 & 8 & 4}
\newcommand{\jmRowWKFive}{\pm10 & \pm3 & \pm1 & \pm11 & \pm3 & \pm13 & \pm4 & \pm1 & \pm12}

\newcommand{\digitsTwoTKFifteen}{90}

\newcommand{\dAKFifteen}{35{,}591}
\newcommand{\dBKFifteen}{35{,}591}
\newcommand{\dCKFifteen}{1}

\newcommand{\hbKFifteen}{-\dfrac{2}{36393}}
\newcommand{\hcKFifteen}{-\dfrac{2}{31657}}
\newcommand{\hdKFifteen}{-\dfrac{2}{38719}}
\newcommand{\hUsedKFifteen}{h_2}
\newcommand{\pzeroSlashKFifteen}{-\dfrac{59010}{35591}}
\newcommand{\qzeroSlashKFifteen}{\dfrac{34890}{35591}}
\newcommand{\MKFifteen}{35{,}591/15}

\newcommand{\jmRowRKFifteen}{37 & 4 & 41 & 26 & 44 & 15 & 45 & 27 & 10}
\newcommand{\jmRowWKFifteen}{\pm11 & \pm6 & \pm12 & \pm20 & \pm7 & \pm21 & \pm5 & \pm26 & \pm20}

\newcommand{\digitsTwoTKHalf}{99}

\newcommand{\dAKHalf}{9{,}581{,}469}
\newcommand{\dBKHalf}{167{,}675{,}707{,}500}
\newcommand{\dCKHalf}{1}

\newcommand{\hbKHalf}{\dfrac{380000}{890729}}
\newcommand{\hcKHalf}{\dfrac{2660000}{7202283}}
\newcommand{\hdKHalf}{\dfrac{140000}{77057}}
\newcommand{\hUsedKHalf}{h_3}
\newcommand{\pzeroSlashKHalf}{-\dfrac{864870}{3193823}}
\newcommand{\qzeroSlashKHalf}{\dfrac{483590}{3193823}}
\newcommand{\MKHalf}{3{,}193{,}823/10}

\newcommand{\jmRowRKHalf}{7 & 25 & 10 & 18 & 32 & 27 & 15 & 16 & 40}
\newcommand{\jmRowWKHalf}{\pm18 & \pm17 & \pm4 & \pm4 & \pm18 & \pm13 & \pm14 & \pm10 & \pm4}

\newcommand{\digitsTwoTKFifth}{97}

\newcommand{\dAKFifth}{4{,}424{,}296}
\newcommand{\dBKFifth}{5{,}184{,}721{,}875}
\newcommand{\dCKFifth}{4}

\newcommand{\hbKFifth}{\dfrac{625}{48021}}
\newcommand{\hcKFifth}{\dfrac{11875}{1602874}}
\newcommand{\hdKFifth}{\dfrac{11875}{826699}}
\newcommand{\hUsedKFifth}{h_1}
\newcommand{\pzeroSlashKFifth}{-\dfrac{38735}{1106074}}
\newcommand{\qzeroSlashKFifth}{\dfrac{155235}{1106074}}
\newcommand{\MKFifth}{4{,}424{,}296/5}

\newcommand{\jmRowRKFifth}{12 & 11 & 17 & 19 & 9 & 13 & 8 & 5 & 16}
\newcommand{\jmRowWKFifth}{\pm7 & \pm3 & \pm10 & \pm1 & \pm10 & \pm5 & \pm4 & \pm9 & \pm7}

\newcommand{\digitsTwoTKFiveThirds}{81}

\newcommand{\dAKFiveThirds}{731{,}851}
\newcommand{\dBKFiveThirds}{533{,}519{,}379}
\newcommand{\dCKFiveThirds}{1}

\newcommand{\hbKFiveThirds}{-\dfrac{162}{97}}
\newcommand{\hcKFiveThirds}{-\dfrac{282042}{165979}}
\newcommand{\hdKFiveThirds}{-\dfrac{282042}{1296613}}
\newcommand{\hUsedKFiveThirds}{h_1}
\newcommand{\pzeroSlashKFiveThirds}{-\dfrac{938290}{731851}}
\newcommand{\qzeroSlashKFiveThirds}{\dfrac{1884390}{731851}}
\newcommand{\MKFiveThirds}{731{,}851/5}

\newcommand{\jmRowRKFiveThirds}{6 & 22 & 13 & 24 & 20 & 16 & 19 & 28 & 4}
\newcommand{\jmRowWKFiveThirds}{\pm12 & \pm3 & \pm11 & \pm4 & \pm1 & \pm2 & \pm7 & \pm2 & \pm3}

\newcommand{\digitsTwoTKFiveQuarters}{93}

\newcommand{\dAKFiveQuarters}{169{,}731}
\newcommand{\dBKFiveQuarters}{86{,}902{,}272}
\newcommand{\dCKFiveQuarters}{1}

\newcommand{\hbKFiveQuarters}{-\dfrac{1120256}{401593}}
\newcommand{\hcKFiveQuarters}{\dfrac{2048}{5309}}
\newcommand{\hdKFiveQuarters}{-\dfrac{1120256}{431649}}
\newcommand{\hUsedKFiveQuarters}{h_2}
\newcommand{\pzeroSlashKFiveQuarters}{-\dfrac{3700750}{56577}}
\newcommand{\qzeroSlashKFiveQuarters}{\dfrac{37570}{56577}}
\newcommand{\MKFiveQuarters}{56{,}577/10}

\newcommand{\jmRowRKFiveQuarters}{27 & 16 & 0 & 20 & 7 & 14 & 24 & 23 & 21}
\newcommand{\jmRowWKFiveQuarters}{\pm3 & \pm4 & \pm10 & \pm13 & \pm12 & \pm13 & \pm2 & \pm10 & \pm9}

\newcommand{\digitsTwoTKElevenSevenths}{149}

\newcommand{\dAKElevenSevenths}{230{,}512{,}803{,}800}
\newcommand{\dBKElevenSevenths}{207{,}547{,}965{,}721{,}425}
\newcommand{\dCKElevenSevenths}{4}

\newcommand{\hbKElevenSevenths}{\dfrac{343}{1143}}
\newcommand{\hcKElevenSevenths}{\dfrac{38627223859}{205795453954}}
\newcommand{\hdKElevenSevenths}{-\dfrac{38627223859}{84376277645}}
\newcommand{\hUsedKElevenSevenths}{h_1}
\newcommand{\pzeroSlashKElevenSevenths}{\dfrac{782008406399}{403397406650}}
\newcommand{\qzeroSlashKElevenSevenths}{\dfrac{3191889047589}{403397406650}}
\newcommand{\MKElevenSevenths}{1{,}613{,}589{,}626{,}600/11}

\newcommand{\jmRowRKElevenSevenths}{27 & 6 & 1 & 9 & 14 & 11 & 15 & 21 & 20}
\newcommand{\jmRowWKElevenSevenths}{\pm14 & \pm4 & \pm2 & \pm1 & \pm11 & \pm12 & \pm8 & \pm8 & \pm11}

\newcommand{\digitsTwoTKTwentyfiveThirteenths}{111}

\newcommand{\dAKTwentyfiveThirteenths}{970{,}713}
\newcommand{\dBKTwentyfiveThirteenths}{27{,}724{,}533{,}993}
\newcommand{\dCKTwentyfiveThirteenths}{1}

\newcommand{\hbKTwentyfiveThirteenths}{-\dfrac{57122}{392059}}
\newcommand{\hcKTwentyfiveThirteenths}{-\dfrac{4394}{4651}}
\newcommand{\hdKTwentyfiveThirteenths}{-\dfrac{57122}{403947}}
\newcommand{\hUsedKTwentyfiveThirteenths}{h_2}
\newcommand{\pzeroSlashKTwentyfiveThirteenths}{-\dfrac{6577700}{4206423}}
\newcommand{\qzeroSlashKTwentyfiveThirteenths}{\dfrac{297200}{4206423}}
\newcommand{\MKTwentyfiveThirteenths}{4{,}206{,}423/25}

\newcommand{\jmRowRKTwentyfiveThirteenths}{13 & 8 & 28 & 5 & 20 & 23 & 25 & 15 & 18}
\newcommand{\jmRowWKTwentyfiveThirteenths}{\pm17 & \pm9 & \pm16 & \pm17 & \pm9 & \pm4 & \pm14 & \pm9 & \pm15}

\newcommand{\digitsTwoTKTwentyTwentyninths}{122}

\newcommand{\dAKTwentyTwentyninths}{6{,}546{,}627}
\newcommand{\dBKTwentyTwentyninths}{134{,}278{,}841{,}844{,}423}
\newcommand{\dCKTwentyTwentyninths}{1}

\newcommand{\hbKTwentyTwentyninths}{\dfrac{5658248}{4544549}}
\newcommand{\hcKTwentyTwentyninths}{\dfrac{5658248}{9942899}}
\newcommand{\hdKTwentyTwentyninths}{\dfrac{5658248}{3375675}}
\newcommand{\hUsedKTwentyTwentyninths}{h_1}
\newcommand{\pzeroSlashKTwentyTwentyninths}{\dfrac{1629200}{2182209}}
\newcommand{\qzeroSlashKTwentyTwentyninths}{\dfrac{4529120}{2182209}}
\newcommand{\MKTwentyTwentyninths}{2{,}182{,}209/80}

\newcommand{\jmRowRKTwentyTwentyninths}{15 & 17 & 22 & 12 & 20 & 25 & 9 & 2 & 26}
\newcommand{\jmRowWKTwentyTwentyninths}{\pm3 & \pm10 & \pm5 & \pm12 & \pm2 & \pm13 & \pm3 & \pm7 & \pm7}

\newcommand{\theoremList}{k&=2, & e&=a+8(b+c+d)\\ k&=3, & e&=a+27(b+c+d)\\ k&=5, & e&=a+125(b+c+d)\\ k&=15, & e&=a+3375(b+c+d)\\ k&=1/2, & 8e&=8a+(b+c+d)\\ k&=1/5, & 125e&=125a+(b+c+d)\\ k&=5/3, & 27e&=27a+125(b+c+d)\\ k&=5/4, & 64e&=64a+125(b+c+d)\\ k&=11/7, & 343e&=343a+1331(b+c+d)\\ k&=25/13, & 2197e&=2197a+15625(b+c+d)\\ k&=20/29, & 24389e&=24389a+8000(b+c+d)}

%% file: generated_constants.tex
\newcommand{\seedCount}{30}

\newcommand{\seedsInGo}{11}
\newcommand{\seedsNew}{19}

\newcommand{\familyCount}{90}
\newcommand{\familyPerSeed}{3}

%% file: generated_seed_table.tex
\begin{table}[htbp]
\caption{The $\seedCount$ known seeds of the relations of Theorem~\ref{thm:main}.}
\label{tab:seeds}
\centering
\scriptsize
\setlength{\arraycolsep}{1.3pt}
\noindent$
\begin{array}{|r|rrrr|r|}
\hline
k & a & b & c & d & \log_{10}e \\
\hline
2 & -258{,}988{,}875 & -231{,}435{,}980 & 127{,}732{,}560 & 174{,}090{,}444 & 8.5 \\
\hline
3 & 13{,}355 & 1{,}530 & -9{,}498 & 8{,}010 & 4.2 \\
3 & 329{,}580 & -32{,}420 & -1{,}911{,}244 & 2{,}018{,}095 & 6.4 \\
3 & 25{,}743{,}730 & 9{,}190{,}029 & 15{,}124{,}680 & -24{,}147{,}360 & 7.5 \\
3 & -124{,}524{,}580 & -52{,}412{,}085 & 75{,}796{,}440 & -13{,}976{,}904 & 8.1 \\
3 & -147{,}189{,}130 & 127{,}707{,}750 & -200{,}914{,}356 & 86{,}874{,}225 & 8.3 \\
3 & 348{,}450{,}105 & -549{,}292{,}210 & -223{,}154{,}564 & 790{,}671{,}260 & 8.9 \\
3 & 1{,}023{,}102{,}610 & 3{,}951{,}204{,}720 & -2{,}297{,}995{,}122 & -1{,}539{,}822{,}855 & 9.6 \\
3 & 11{,}018{,}275{,}120 & -8{,}201{,}758{,}680 & 8{,}419{,}634{,}805 & -163{,}591{,}146 & 10.1 \\
3 & -40{,}514{,}975{,}437{,}850 & -37{,}987{,}256{,}480{,}120 & -107{,}503{,}232{,}746{,}244 & 152{,}985{,}292{,}108{,}385 & 14.2 \\
\hline
5 & 571{,}994 & 67{,}965 & -2{,}335{,}290 & 2{,}284{,}740 & 6.4 \\
5 & 634{,}015{,}318 & -362{,}998{,}270 & 835{,}952{,}135 & -470{,}666{,}060 & 9.0 \\
\hline
15 & -35{,}591 & 12{,}060 & -58{,}980 & 46{,}950 & 4.8 \\
\hline
1/2 & 3{,}193{,}823 & -190{,}640 & -674{,}230 & 2{,}194{,}870 & 6.5 \\
\hline
1/5 & 1{,}106{,}074 & 41{,}110 & -96{,}985 & 58{,}250 & 6.0 \\
1/5 & 12{,}617{,}448 & 1{,}830{,}650 & 737{,}925 & 4{,}614{,}300 & 7.1 \\
1/5 & 162{,}727{,}467{,}264{,}452 & 36{,}054{,}239{,}976{,}640 & -34{,}698{,}506{,}490{,}485 & 23{,}624{,}101{,}627{,}470 & 14.2 \\
\hline
5/3 & -731{,}851 & -468{,}220 & -473{,}050 & 1{,}411{,}340 & 6.2 \\
5/3 & -667{,}474 & -153{,}145 & -4{,}655{,}140 & 6{,}427{,}880 & 6.8 \\
5/3 & -26{,}324{,}398 & 101{,}329{,}560 & -89{,}425{,}715 & 18{,}445{,}370 & 8.1 \\
5/3 & -18{,}747{,}702{,}922 & -4{,}329{,}665{,}650 & 31{,}327{,}031{,}400 & -15{,}873{,}707{,}705 & 10.5 \\
5/3 & -792{,}735{,}389 & -47{,}861{,}125{,}920 & 5{,}484{,}360{,}380 & 56{,}010{,}515{,}710 & 10.8 \\
5/3 & \begin{array}[t]{@{}r@{}}-2{,}983{,}335{,}613{,}979\\960{,}189\end{array} & \begin{array}[t]{@{}r@{}}-5{,}912{,}201{,}115{,}326\\921{,}980\end{array} & \begin{array}[t]{@{}r@{}}7{,}639{,}689{,}166{,}079\\778{,}250\end{array} & \begin{array}[t]{@{}r@{}}706{,}101{,}284{,}037{,}470\\010\end{array} & 18.9 \\
5/3 & \begin{array}[t]{@{}r@{}}512{,}913{,}997{,}537{,}775\\926{,}570{,}712\end{array} & \begin{array}[t]{@{}r@{}}830{,}145{,}005{,}167{,}235\\140{,}736{,}610\end{array} & \begin{array}[t]{@{}r@{}}388{,}372{,}190{,}242{,}411\\536{,}691{,}415\end{array} & \begin{array}[t]{@{}r@{}}-1{,}075{,}542{,}606{,}191\\524{,}985{,}472{,}260\end{array} & 24.1 \\
5/3 & \begin{array}[t]{@{}r@{}}2{,}098{,}384{,}979{,}159\\322{,}754{,}051{,}328\end{array} & \begin{array}[t]{@{}r@{}}16{,}679{,}409{,}001{,}735\\096{,}666{,}245{,}070\end{array} & \begin{array}[t]{@{}r@{}}-6{,}767{,}827{,}700{,}115\\360{,}279{,}350{,}640\end{array} & \begin{array}[t]{@{}r@{}}-6{,}714{,}724{,}544{,}272\\328{,}428{,}304{,}845\end{array} & 25.2 \\
\hline
5/4 & -56{,}577 & 1{,}831{,}590 & -2{,}300{,}430 & 1{,}869{,}160 & 6.4 \\
5/4 & -5{,}938{,}786{,}023 & -87{,}349{,}755{,}660 & 75{,}881{,}967{,}700 & 67{,}144{,}321{,}080 & 11.0 \\
\hline
11/7 & 403{,}397{,}406{,}650 & -357{,}108{,}943{,}950 & -1{,}204{,}940{,}320{,}595 & 1{,}986{,}948{,}726{,}994 & 12.3 \\
\hline
25/13 & -4{,}206{,}423 & 3{,}140{,}250 & -5{,}149{,}650 & 3{,}437{,}450 & 6.8 \\
\hline
20/29 & 2{,}182{,}209 & 2{,}273{,}040 & -1{,}449{,}960 & 3{,}079{,}160 & 6.5 \\
\hline
\end{array}$
\end{table}

%% file: generated_seed_data.tex
\begin{table}[ht]
\caption{The seed and placement of each family.}
\label{tab:seeddata}
\centering\scriptsize\setlength{\arraycolsep}{2.3pt}\renewcommand{\arraystretch}{2.9}
\noindent$\begin{array}{|l|r|r|r|r|c|r|r|}
\hline
k & e & h_1 & h_2 & h_3 & h & p_0 & q_0 \\ \hline
2 & {\renewcommand{\arraystretch}{1}\begin{array}[t]{@{}r@{}}304{,}107{,}317\end{array}} & \hbKTwo & \hcKTwo & \hdKTwo & \hUsedKTwo & \pzeroSlashKTwo & \qzeroSlashKTwo \\
3 & {\renewcommand{\arraystretch}{1}\begin{array}[t]{@{}r@{}}14{,}489\end{array}} & \hbKThree & \hcKThree & \hdKThree & \hUsedKThree & \pzeroSlashKThree & \qzeroSlashKThree \\
5 & {\renewcommand{\arraystretch}{1}\begin{array}[t]{@{}r@{}}2{,}748{,}869\end{array}} & \hbKFive & \hcKFive & \hdKFive & \hUsedKFive & \pzeroSlashKFive & \qzeroSlashKFive \\
15 & {\renewcommand{\arraystretch}{1}\begin{array}[t]{@{}r@{}}65{,}659\end{array}} & \hbKFifteen & \hcKFifteen & \hdKFifteen & \hUsedKFifteen & \pzeroSlashKFifteen & \qzeroSlashKFifteen \\
1/2 & {\renewcommand{\arraystretch}{1}\begin{array}[t]{@{}r@{}}3{,}360{,}073\end{array}} & \hbKHalf & \hcKHalf & \hdKHalf & \hUsedKHalf & \pzeroSlashKHalf & \qzeroSlashKHalf \\
1/5 & {\renewcommand{\arraystretch}{1}\begin{array}[t]{@{}r@{}}1{,}106{,}093\end{array}} & \hbKFifth & \hcKFifth & \hdKFifth & \hUsedKFifth & \pzeroSlashKFifth & \qzeroSlashKFifth \\
5/3 & {\renewcommand{\arraystretch}{1}\begin{array}[t]{@{}r@{}}1{,}444{,}399\end{array}} & \hbKFiveThirds & \hcKFiveThirds & \hdKFiveThirds & \hUsedKFiveThirds & \pzeroSlashKFiveThirds & \qzeroSlashKFiveThirds \\
5/4 & {\renewcommand{\arraystretch}{1}\begin{array}[t]{@{}r@{}}2{,}678{,}423\end{array}} & \hbKFiveQuarters & \hcKFiveQuarters & \hdKFiveQuarters & \hUsedKFiveQuarters & \pzeroSlashKFiveQuarters & \qzeroSlashKFiveQuarters \\
11/7 & {\renewcommand{\arraystretch}{1}\begin{array}[t]{@{}r@{}}2{,}052{,}205\\524{,}783\end{array}} & \hbKElevenSevenths & \hcKElevenSevenths & \hdKElevenSevenths & \hUsedKElevenSevenths & \pzeroSlashKElevenSevenths & \qzeroSlashKElevenSevenths \\
25/13 & {\renewcommand{\arraystretch}{1}\begin{array}[t]{@{}r@{}}5{,}949{,}827\end{array}} & \hbKTwentyfiveThirteenths & \hcKTwentyfiveThirteenths & \hdKTwentyfiveThirteenths & \hUsedKTwentyfiveThirteenths & \pzeroSlashKTwentyfiveThirteenths & \qzeroSlashKTwentyfiveThirteenths \\
20/29 & {\renewcommand{\arraystretch}{1}\begin{array}[t]{@{}r@{}}3{,}462{,}209\end{array}} & \hbKTwentyTwentyninths & \hcKTwentyTwentyninths & \hdKTwentyTwentyninths & \hUsedKTwentyTwentyninths & \pzeroSlashKTwentyTwentyninths & \qzeroSlashKTwentyTwentyninths \\
\hline
\end{array}$
\end{table}

%% file: generated_families.tex
\newlength{\cwa}\settowidth{\cwa}{\scriptsize$983{,}911{,}540$}
\newlength{\cwb}\settowidth{\cwb}{\scriptsize$-64{,}799{,}180$}
\newlength{\cwc}\settowidth{\cwc}{\scriptsize$028{,}025{,}981$}
\newlength{\cwd}\settowidth{\cwd}{\scriptsize$-193{,}615{,}006$}
\newlength{\cwe}\settowidth{\cwe}{\scriptsize$-121{,}956{,}838$}
\newlength{\cwf}\settowidth{\cwf}{\scriptsize$9{,}910{,}191$}
\newlength{\cwg}\settowidth{\cwg}{\scriptsize$205{,}061{,}217$}

\begin{table}[H]
\caption{The polynomials \eqref{eq:ABCform} of each family.}
\label{tab:coef}
\vskip\belowcaptionskip
\centering\scriptsize\setlength{\arraycolsep}{2pt}\renewcommand{\arraystretch}{1.15}
\noindent$\begin{array}{|c|w{r}{\cwa}|w{r}{\cwb}|w{r}{\cwc}|w{r}{\cwd}|w{r}{\cwe}|w{r}{\cwf}|w{r}{\cwg}|}
\hline
k & a_0 & a_1 & b_0 & b_1 & b_2 & b_3 & c_0 \\ \hline
2 & \begin{array}[t]{@{}r@{}}5{,}322{,}235\\983{,}911{,}540\\704\end{array} & \begin{array}[t]{@{}r@{}}-4{,}761{,}962\\937{,}888\end{array} & \begin{array}[t]{@{}r@{}}2{,}831{,}288\\028{,}025{,}981\\373{,}962{,}506\end{array} & \begin{array}[t]{@{}r@{}}957{,}745{,}284\\562{,}998{,}450\end{array} & \begin{array}[t]{@{}r@{}}-121{,}956{,}838\\650\end{array} & \begin{array}[t]{@{}r@{}}-475{,}588\end{array} & \begin{array}[t]{@{}r@{}}3{,}752{,}195\\950{,}764\end{array} \\
\cline{2-8}
 & \multicolumn{2}{c|}{d_A=\dAKTwo} & \multicolumn{4}{c|}{d_B=\dBKTwo} & d_C=\dCKTwo \\
\hline
3 & \begin{array}[t]{@{}r@{}}48{,}044{,}056\\139\end{array} & \begin{array}[t]{@{}r@{}}1{,}242{,}748\end{array} & \begin{array}[t]{@{}r@{}}421{,}773{,}010\\904{,}633{,}577\end{array} & \begin{array}[t]{@{}r@{}}11{,}850{,}222\\510{,}380\end{array} & \begin{array}[t]{@{}r@{}}495{,}920\end{array} & \begin{array}[t]{@{}r@{}}56\end{array} & \begin{array}[t]{@{}r@{}}6{,}190{,}029\\774\end{array} \\
\cline{2-8}
 & \multicolumn{2}{c|}{d_A=\dAKThree} & \multicolumn{4}{c|}{d_B=\dBKThree} & d_C=\dCKThree \\
\hline
5 & \begin{array}[t]{@{}r@{}}-48{,}876{,}162\\384{,}454{,}515\\045\end{array} & \begin{array}[t]{@{}r@{}}1{,}351{,}240\\358{,}130\end{array} & \begin{array}[t]{@{}r@{}}-1{,}138{,}944\\450{,}855{,}262\\927{,}386{,}545\\190\end{array} & \begin{array}[t]{@{}r@{}}48{,}120{,}579\\186{,}390{,}209\\655\end{array} & \begin{array}[t]{@{}r@{}}-442{,}339{,}874\\730\end{array} & \begin{array}[t]{@{}r@{}}1{,}720\end{array} & \begin{array}[t]{@{}r@{}}7{,}551{,}957\\205{,}061{,}217\end{array} \\
\cline{2-8}
 & \multicolumn{2}{c|}{d_A=\dAKFive} & \multicolumn{4}{c|}{d_B=\dBKFive} & d_C=\dCKFive \\
\hline
15 & \begin{array}[t]{@{}r@{}}988{,}379{,}273\\429{,}115\end{array} & \begin{array}[t]{@{}r@{}}-64{,}799{,}180\end{array} & \begin{array}[t]{@{}r@{}}3{,}881{,}557\\944{,}452{,}410\\291{,}072{,}705\end{array} & \begin{array}[t]{@{}r@{}}-193{,}615{,}006\\577{,}245{,}180\end{array} & \begin{array}[t]{@{}r@{}}-3{,}476{,}962\\420\end{array} & \begin{array}[t]{@{}r@{}}-40\end{array} & \begin{array}[t]{@{}r@{}}227{,}764{,}111\\207{,}280\end{array} \\
\cline{2-8}
 & \multicolumn{2}{c|}{d_A=\dAKFifteen} & \multicolumn{4}{c|}{d_B=\dBKFifteen} & d_C=\dCKFifteen \\
\hline
1/2 & \begin{array}[t]{@{}r@{}}35{,}644{,}173\\121{,}640{,}345\\150{,}000{,}000\end{array} & \begin{array}[t]{@{}r@{}}4{,}564{,}478\\282{,}000{,}000\end{array} & \begin{array}[t]{@{}r@{}}1{,}861{,}839\\452{,}667{,}367\\460{,}834{,}286\\853{,}750\end{array} & \begin{array}[t]{@{}r@{}}-774{,}697{,}658\\477{,}674{,}041\\950\end{array} & \begin{array}[t]{@{}r@{}}34{,}759{,}598\\785\end{array} & \begin{array}[t]{@{}r@{}}19\end{array} & \begin{array}[t]{@{}r@{}}2{,}456{,}875\\951{,}806{,}010\\000\end{array} \\
\cline{2-8}
 & \multicolumn{2}{c|}{d_A=\dAKHalf} & \multicolumn{4}{c|}{d_B=\dBKHalf} & d_C=\dCKHalf \\
\hline
1/5 & \begin{array}[t]{@{}r@{}}89{,}222{,}486\\669{,}224{,}296\\875\end{array} & \begin{array}[t]{@{}r@{}}783{,}966{,}718\\750\end{array} & \begin{array}[t]{@{}r@{}}4{,}236{,}231\\885{,}715{,}864\\751{,}534{,}820\\000\end{array} & \begin{array}[t]{@{}r@{}}51{,}358{,}670\\356{,}869{,}950\\625\end{array} & \begin{array}[t]{@{}r@{}}18{,}172{,}400\\550\end{array} & \begin{array}[t]{@{}r@{}}152\end{array} & \begin{array}[t]{@{}r@{}}135{,}742{,}527\\834{,}740{,}625\end{array} \\
\cline{2-8}
 & \multicolumn{2}{c|}{d_A=\dAKFifth} & \multicolumn{4}{c|}{d_B=\dBKFifth} & d_C=\dCKFifth \\
\hline
5/3 & \begin{array}[t]{@{}r@{}}2{,}213{,}752\\851{,}633{,}735\end{array} & \begin{array}[t]{@{}r@{}}-14{,}404{,}345\\020\end{array} & \begin{array}[t]{@{}r@{}}320{,}505{,}192\\996{,}143{,}830\\605\end{array} & \begin{array}[t]{@{}r@{}}702{,}958{,}627\\463{,}220\end{array} & \begin{array}[t]{@{}r@{}}4{,}763{,}676\\780\end{array} & \begin{array}[t]{@{}r@{}}-69{,}640\end{array} & \begin{array}[t]{@{}r@{}}38{,}999{,}364\\840\end{array} \\
\cline{2-8}
 & \multicolumn{2}{c|}{d_A=\dAKFiveThirds} & \multicolumn{4}{c|}{d_B=\dBKFiveThirds} & d_C=\dCKFiveThirds \\
\hline
5/4 & \begin{array}[t]{@{}r@{}}166{,}380{,}097\\234{,}032{,}394\\240\end{array} & \begin{array}[t]{@{}r@{}}-15{,}383{,}611\\228{,}160\end{array} & \begin{array}[t]{@{}r@{}}3{,}483{,}760\\390{,}420{,}177\\356{,}423{,}040\end{array} & \begin{array}[t]{@{}r@{}}-964{,}618{,}399\\951{,}008{,}200\end{array} & \begin{array}[t]{@{}r@{}}88{,}987{,}960\\930\end{array} & \begin{array}[t]{@{}r@{}}-2{,}735\end{array} & \begin{array}[t]{@{}r@{}}117{,}167{,}101\\219{,}840\end{array} \\
\cline{2-8}
 & \multicolumn{2}{c|}{d_A=\dAKFiveQuarters} & \multicolumn{4}{c|}{d_B=\dBKFiveQuarters} & d_C=\dCKFiveQuarters \\
\hline
11/7 & \begin{array}[t]{@{}r@{}}-50{,}249{,}959\\640{,}029{,}235\\195{,}109\end{array} & \begin{array}[t]{@{}r@{}}79{,}260{,}572\\244{,}506{,}422\end{array} & \begin{array}[t]{@{}r@{}}-8{,}436{,}633\\864{,}108{,}747\\185{,}556{,}294\\462\end{array} & \begin{array}[t]{@{}r@{}}16{,}321{,}044\\841{,}585{,}509\\420{,}945\end{array} & \begin{array}[t]{@{}r@{}}-15{,}068{,}389\\844{,}289{,}210\end{array} & \begin{array}[t]{@{}r@{}}9{,}910{,}191\\544\end{array} & \begin{array}[t]{@{}r@{}}2{,}217{,}818\\683{,}161\end{array} \\
\cline{2-8}
 & \multicolumn{2}{c|}{d_A=\dAKElevenSevenths} & \multicolumn{4}{c|}{d_B=\dBKElevenSevenths} & d_C=\dCKElevenSevenths \\
\hline
25/13 & \begin{array}[t]{@{}r@{}}2{,}599{,}452\\361{,}144{,}313\\605{,}775\end{array} & \begin{array}[t]{@{}r@{}}-16{,}885{,}276\\162{,}300\end{array} & \begin{array}[t]{@{}r@{}}70{,}068{,}258\\365{,}250{,}698\\508{,}925{,}650\\925\end{array} & \begin{array}[t]{@{}r@{}}178{,}386{,}053\\488{,}054{,}387\\300\end{array} & \begin{array}[t]{@{}r@{}}216{,}094{,}210\\600\end{array} & \begin{array}[t]{@{}r@{}}-2{,}600\end{array} & \begin{array}[t]{@{}r@{}}118{,}431{,}130\\473{,}269{,}122\end{array} \\
\cline{2-8}
 & \multicolumn{2}{c|}{d_A=\dAKTwentyfiveThirteenths} & \multicolumn{4}{c|}{d_B=\dBKTwentyfiveThirteenths} & d_C=\dCKTwentyfiveThirteenths \\
\hline
20/29 & \begin{array}[t]{@{}r@{}}-30{,}079{,}447\\433{,}713{,}550\\379{,}522{,}106\\182{,}080\end{array} & \begin{array}[t]{@{}r@{}}29{,}203{,}574\\879{,}755{,}833\\280\end{array} & \begin{array}[t]{@{}r@{}}-3{,}535{,}159\\512{,}146{,}778\\820{,}060{,}440\\801{,}005{,}856\\912{,}260\end{array} & \begin{array}[t]{@{}r@{}}3{,}069{,}173\\984{,}617{,}167\\314{,}212{,}934\\020\end{array} & \begin{array}[t]{@{}r@{}}-621{,}425{,}006\\004{,}680\end{array} & \begin{array}[t]{@{}r@{}}40\end{array} & \begin{array}[t]{@{}r@{}}54{,}608{,}891\\906{,}149{,}335\\932{,}637{,}700\end{array} \\
\cline{2-8}
 & \multicolumn{2}{c|}{d_A=\dAKTwentyTwentyninths} & \multicolumn{4}{c|}{d_B=\dBKTwentyTwentyninths} & d_C=\dCKTwentyTwentyninths \\
\hline
\end{array}$
\end{table}

%% file: generated_curve_table.tex
\newlength{\nwa}\settowidth{\nwa}{\scriptsize$-353{,}330{,}176$}
\newlength{\nwb}\settowidth{\nwb}{\scriptsize$-477{,}656{,}413$}
\newlength{\nwc}\settowidth{\nwc}{\scriptsize$-115{,}416{,}416$}
\newlength{\nwd}\settowidth{\nwd}{\scriptsize$619{,}971{,}175$}
\newlength{\nwe}\settowidth{\nwe}{\scriptsize$-153{,}177{,}390$}

\begin{table}[H]
\caption{The curve of each family in the form \eqref{eq:Nform}.}
\label{tab:curves}
\vskip\belowcaptionskip
\centering\scriptsize\setlength{\arraycolsep}{2pt}\renewcommand{\arraystretch}{1.0}
\noindent$\begin{array}{|l|c|w{r}{\nwa}|w{r}{\nwb}|w{r}{\nwc}|w{r}{\nwd}|w{r}{\nwe}|}
\hline
k & M & n_4 & n_3 & n_2 & n_1 & n_0 \\ \hline
2 & \MKTwo & \begin{array}[t]{@{}r@{}}-3{,}810{,}637\\554{,}655{,}983\end{array} & \begin{array}[t]{@{}r@{}}3{,}294{,}658\\144{,}329{,}788\\906{,}356\end{array} & \begin{array}[t]{@{}r@{}}8{,}688{,}394\\158{,}150{,}751\\016{,}404{,}499\\928\end{array} & \begin{array}[t]{@{}r@{}}14{,}164{,}091\\619{,}971{,}175\\258{,}179{,}729\\003{,}808{,}984\end{array} & \begin{array}[t]{@{}r@{}}-25{,}243{,}911\\016{,}257{,}649\\955{,}470{,}405\\138{,}570{,}884\\802{,}668\end{array} \\
\hline
3 & \MKThree & \begin{array}[t]{@{}r@{}}-35{,}534{,}992\end{array} & \begin{array}[t]{@{}r@{}}-3{,}306{,}770\\731{,}944\end{array} & \begin{array}[t]{@{}r@{}}-15{,}172{,}317\\493{,}269{,}316\\128\end{array} & \begin{array}[t]{@{}r@{}}-1{,}093{,}490\\321{,}304{,}049\\798{,}772{,}416\end{array} & \begin{array}[t]{@{}r@{}}-18{,}958{,}669\\594{,}580{,}211\\381{,}729{,}967\\107\end{array} \\
\hline
5 & \MKFive & \begin{array}[t]{@{}r@{}}-2{,}711{,}140\\714{,}864\end{array} & \begin{array}[t]{@{}r@{}}897{,}810{,}028\\044{,}913{,}649\\952\end{array} & \begin{array}[t]{@{}r@{}}-115{,}416{,}416\\677{,}615{,}435\\427{,}251{,}232\\520\end{array} & \begin{array}[t]{@{}r@{}}5{,}123{,}867\\588{,}492{,}973\\405{,}625{,}705\\790{,}324{,}846\\952\end{array} & \begin{array}[t]{@{}r@{}}-72{,}061{,}042\\193{,}988{,}843\\770{,}499{,}789\\690{,}575{,}543\\249{,}681{,}727\end{array} \\
\hline
15 & \MKFifteen & \begin{array}[t]{@{}r@{}}-98{,}268{,}412\end{array} & \begin{array}[t]{@{}r@{}}-7{,}430{,}765\\710{,}786{,}040\end{array} & \begin{array}[t]{@{}r@{}}-366{,}845{,}976\\268{,}258{,}433\\382{,}324\end{array} & \begin{array}[t]{@{}r@{}}17{,}715{,}473\\257{,}951{,}705\\951{,}344{,}852\\391{,}904\end{array} & \begin{array}[t]{@{}r@{}}-153{,}177{,}390\\705{,}708{,}590\\080{,}259{,}709\\604{,}041{,}593\\843\end{array} \\
\hline
1/2 & \MKHalf & \begin{array}[t]{@{}r@{}}-47{,}218{,}599\\895\end{array} & \begin{array}[t]{@{}r@{}}-477{,}656{,}413\\175{,}047{,}225\\840\end{array} & \begin{array}[t]{@{}r@{}}1{,}301{,}145\\009{,}354{,}239\\559{,}197{,}976\\829{,}800\end{array} & \begin{array}[t]{@{}r@{}}10{,}922{,}932\\405{,}392{,}374\\290{,}231{,}179\\106{,}910{,}625\\020{,}000\end{array} & \begin{array}[t]{@{}r@{}}-37{,}908{,}013\\850{,}883{,}469\\127{,}529{,}591\\691{,}935{,}543\\873{,}996{,}226\\750{,}000\end{array} \\
\hline
1/5 & \MKFifth & \begin{array}[t]{@{}r@{}}-21{,}184{,}196\\208\end{array} & \begin{array}[t]{@{}r@{}}-8{,}542{,}732\\862{,}888{,}843\\808\end{array} & \begin{array}[t]{@{}r@{}}-13{,}913{,}788\\996{,}424{,}823\\614{,}900{,}872\\200\end{array} & \begin{array}[t]{@{}r@{}}-2{,}427{,}928\\502{,}038{,}642\\395{,}817{,}110\\111{,}622{,}185\\000\end{array} & \begin{array}[t]{@{}r@{}}-98{,}350{,}353\\606{,}460{,}216\\609{,}581{,}519\\111{,}478{,}610\\219{,}139{,}375\end{array} \\
\hline
5/3 & \MKFiveThirds & \begin{array}[t]{@{}r@{}}-353{,}330{,}176\\668\end{array} & \begin{array}[t]{@{}r@{}}110{,}016{,}490\\118{,}842{,}440\end{array} & \begin{array}[t]{@{}r@{}}-11{,}286{,}087\\640{,}387{,}386\\615{,}396\end{array} & \begin{array}[t]{@{}r@{}}1{,}511{,}353\\439{,}254{,}767\\213{,}413{,}941\\536\end{array} & \begin{array}[t]{@{}r@{}}-134{,}348{,}612\\725{,}430{,}751\\586{,}658{,}267\\089{,}827\end{array} \\
\hline
5/4 & \MKFiveQuarters & \begin{array}[t]{@{}r@{}}-821{,}747{,}148\\799\end{array} & \begin{array}[t]{@{}r@{}}35{,}625{,}112\\516{,}970{,}507\\696\end{array} & \begin{array}[t]{@{}r@{}}-579{,}009{,}949\\056{,}088{,}466\\385{,}863{,}616\end{array} & \begin{array}[t]{@{}r@{}}4{,}181{,}369\\504{,}606{,}882\\791{,}265{,}152\\581{,}773{,}312\end{array} & \begin{array}[t]{@{}r@{}}-11{,}320{,}673\\267{,}757{,}974\\580{,}383{,}237\\861{,}678{,}268\\203{,}008\end{array} \\
\hline
11/7 & \MKElevenSevenths & \begin{array}[t]{@{}r@{}}-1{,}832{,}380\\638{,}438{,}457\\490{,}458{,}992\end{array} & \begin{array}[t]{@{}r@{}}6{,}850{,}325\\897{,}679{,}733\\437{,}813{,}195\\907{,}296\end{array} & \begin{array}[t]{@{}r@{}}-9{,}795{,}355\\204{,}412{,}187\\976{,}445{,}434\\145{,}747{,}733\\768\end{array} & \begin{array}[t]{@{}r@{}}6{,}026{,}611\\998{,}486{,}435\\849{,}186{,}703\\490{,}744{,}963\\618{,}154{,}264\end{array} & \begin{array}[t]{@{}r@{}}-1{,}301{,}914\\460{,}957{,}137\\545{,}579{,}351\\406{,}226{,}325\\556{,}686{,}273\\867{,}327\end{array} \\
\hline
25/13 & \MKTwentyfiveThirteenths & \begin{array}[t]{@{}r@{}}-415{,}496{,}242\\528\end{array} & \begin{array}[t]{@{}r@{}}98{,}531{,}470\\409{,}755{,}563\\128\end{array} & \begin{array}[t]{@{}r@{}}23{,}165{,}305\\863{,}062{,}161\\259{,}981{,}238\\744\end{array} & \begin{array}[t]{@{}r@{}}6{,}811{,}037\\583{,}658{,}285\\341{,}859{,}975\\266{,}521{,}581\\808\end{array} & \begin{array}[t]{@{}r@{}}-1{,}722{,}411\\634{,}290{,}991\\242{,}550{,}962\\016{,}473{,}591\\947{,}984{,}338\\267\end{array} \\
\hline
20/29 & \MKTwentyTwentyninths & \begin{array}[t]{@{}r@{}}-5{,}693{,}544\\346\end{array} & \begin{array}[t]{@{}r@{}}147{,}412{,}280\\453{,}234{,}640\\220{,}676\end{array} & \begin{array}[t]{@{}r@{}}-1{,}175{,}242\\470{,}894{,}898\\524{,}056{,}620\\277{,}474{,}595\\225\end{array} & \begin{array}[t]{@{}r@{}}1{,}489{,}726\\034{,}637{,}863\\303{,}318{,}997\\292{,}311{,}155\\427{,}265{,}556\\919{,}854\end{array} & \begin{array}[t]{@{}r@{}}8{,}748{,}117\\941{,}925{,}992\\133{,}515{,}030\\570{,}436{,}704\\375{,}964{,}911\\681{,}868{,}101\\060{,}606{,}597\end{array} \\
\hline
\end{array}$
\end{table}

%% file: generated_appendix.tex
\begin{table}[H]
\caption{The certificate that $T$ has infinite order.}
\label{tab:cert}
\vskip\belowcaptionskip
\centering
\noindent$
\begin{array}{|l|l|r|ccccccccc|}
\hline
k & \ell & j & 1 & 2 & 3 & 4 & 5 & 6 & 7 & 8 & 9 \\
\hline
2 & 47 & r & \jmRowRKTwo \\
 & & w & \jmRowWKTwo \\
\hline
3 & 31 & r & \jmRowRKThree \\
 & & w & \jmRowWKThree \\
\hline
5 & 29 & r & \jmRowRKFive \\
 & & w & \jmRowWKFive \\
\hline
15 & 53 & r & \jmRowRKFifteen \\
 & & w & \jmRowWKFifteen \\
\hline
1/2 & 41 & r & \jmRowRKHalf \\
 & & w & \jmRowWKHalf \\
\hline
1/5 & 23 & r & \jmRowRKFifth \\
 & & w & \jmRowWKFifth \\
\hline
5/3 & 29 & r & \jmRowRKFiveThirds \\
 & & w & \jmRowWKFiveThirds \\
\hline
5/4 & 31 & r & \jmRowRKFiveQuarters \\
 & & w & \jmRowWKFiveQuarters \\
\hline
11/7 & 31 & r & \jmRowRKElevenSevenths \\
 & & w & \jmRowWKElevenSevenths \\
\hline
25/13 & 37 & r & \jmRowRKTwentyfiveThirteenths \\
 & & w & \jmRowWKTwentyfiveThirteenths \\
\hline
20/29 & 31 & r & \jmRowRKTwentyTwentyninths \\
 & & w & \jmRowWKTwentyTwentyninths \\
\hline
\end{array}$
\end{table}

%% file: generated_solutions.tex
\needspace{8\baselineskip}
For $k=2$, the point $2T$ gives the solution with $\digitsTwoTKTwo$ digits in $e$:
{\scriptsize\setlength{\arraycolsep}{4pt}\[\begin{array}{rrrrr}
a & b & c & d & e \\ \hline
\begin{array}[t]{@{}r@{}}-101{,}572{,}248{,}705\\808{,}030{,}152{,}163\\714{,}289{,}572{,}689\\257{,}384{,}340{,}709\\886{,}564{,}394{,}805\\382{,}812{,}727{,}502\\823{,}146{,}675{,}261\\832{,}083{,}523{,}740\\921{,}019{,}875{,}195\end{array} & \begin{array}[t]{@{}r@{}}30{,}717{,}671{,}736\\482{,}869{,}357{,}533\\521{,}259{,}765{,}856\\371{,}632{,}882{,}573\\703{,}157{,}752{,}054\\468{,}636{,}168{,}949\\518{,}467{,}770{,}918\\032{,}237{,}463{,}074\\461{,}000{,}561{,}508\end{array} & \begin{array}[t]{@{}r@{}}86{,}374{,}003{,}479\\275{,}583{,}942{,}363\\611{,}382{,}102{,}143\\101{,}248{,}529{,}658\\629{,}508{,}192{,}894\\665{,}310{,}783{,}334\\025{,}260{,}461{,}314\\320{,}471{,}366{,}332\\409{,}524{,}591{,}400\end{array} & \begin{array}[t]{@{}r@{}}-89{,}070{,}074{,}148\\680{,}508{,}664{,}791\\117{,}683{,}194{,}617\\518{,}647{,}493{,}993\\078{,}125{,}788{,}711\\682{,}927{,}329{,}834\\982{,}276{,}884{,}581\\936{,}656{,}428{,}361\\555{,}268{,}474{,}300\end{array} & \begin{array}[t]{@{}r@{}}122{,}600{,}559{,}830\\815{,}526{,}928{,}684\\405{,}379{,}814{,}366\\376{,}487{,}005{,}204\\149{,}756{,}855{,}094\\225{,}344{,}252{,}085\\668{,}464{,}105{,}941\\496{,}335{,}684{,}621\\601{,}033{,}553{,}669\end{array} \\
\end{array}\]}

\needspace{8\baselineskip}
For $k=3$, the point $2T$ gives the solution with $\digitsTwoTKThree$ digits in $e$:
{\scriptsize\setlength{\arraycolsep}{4pt}\[\begin{array}{rrrrr}
a & b & c & d & e \\ \hline
\begin{array}[t]{@{}r@{}}8{,}272{,}712{,}149\\792{,}361{,}910{,}248\\019{,}597{,}100{,}048\\598{,}000{,}110{,}765\\711{,}467{,}076{,}594\\710{,}346{,}856{,}306\\745\end{array} & \begin{array}[t]{@{}r@{}}239{,}493{,}847{,}126\\549{,}407{,}958{,}968\\328{,}130{,}290{,}499\\445{,}192{,}518{,}012\\680{,}902{,}042{,}248\\033{,}492{,}180{,}186\end{array} & \begin{array}[t]{@{}r@{}}5{,}790{,}148{,}342\\998{,}764{,}209{,}103\\630{,}632{,}361{,}697\\813{,}124{,}706{,}175\\075{,}721{,}469{,}766\\480{,}126{,}333{,}010\\550\end{array} & \begin{array}[t]{@{}r@{}}-5{,}996{,}057{,}237\\877{,}624{,}559{,}525\\520{,}400{,}065{,}431\\824{,}702{,}103{,}123\\148{,}040{,}038{,}001\\876{,}924{,}985{,}349\\570\end{array} & \begin{array}[t]{@{}r@{}}9{,}179{,}505{,}860\\479{,}966{,}463{,}749\\140{,}728{,}617{,}073\\770{,}430{,}591{,}154\\101{,}250{,}089{,}379\\693{,}687{,}532{,}018\\227\end{array} \\
\end{array}\]}

\needspace{8\baselineskip}
For $k=5$, the point $2T$ gives the solution with $\digitsTwoTKFive$ digits in $e$:
{\scriptsize\setlength{\arraycolsep}{4pt}\[\begin{array}{rrrrr}
a & b & c & d & e \\ \hline
\begin{array}[t]{@{}r@{}}19{,}681{,}889{,}714\\839{,}531{,}370{,}025\\687{,}744{,}195{,}707\\692{,}032{,}956{,}643\\923{,}776{,}002{,}141\\071{,}009{,}825{,}607\\469{,}128{,}826{,}531\\417{,}273{,}223{,}317\\814\end{array} & \begin{array}[t]{@{}r@{}}50{,}345{,}160{,}123\\760{,}788{,}111{,}861\\086{,}435{,}466{,}844\\600{,}642{,}083{,}335\\332{,}720{,}287{,}726\\669{,}850{,}388{,}729\\907{,}012{,}682{,}982\\621{,}275{,}685{,}713\\215\end{array} & \begin{array}[t]{@{}r@{}}-51{,}069{,}393{,}816\\663{,}927{,}123{,}308\\730{,}129{,}112{,}903\\916{,}919{,}707{,}758\\355{,}512{,}276{,}169\\736{,}429{,}717{,}128\\587{,}232{,}431{,}192\\555{,}949{,}247{,}061\\860\end{array} & \begin{array}[t]{@{}r@{}}1{,}050{,}589{,}741\\795{,}069{,}551{,}556\\249{,}360{,}679{,}011\\711{,}025{,}195{,}462\\928{,}928{,}640{,}528\\166{,}408{,}821{,}233\\109{,}252{,}512{,}703\\241{,}305{,}324{,}474\\010\end{array} & \begin{array}[t]{@{}r@{}}60{,}476{,}395{,}826\\330{,}848{,}883{,}601\\396{,}123{,}314{,}757\\035{,}479{,}336{,}632\\190{,}857{,}512{,}778\\549{,}696{,}429{,}911\\098{,}224{,}388{,}194\\746{,}243{,}613{,}988\\439\end{array} \\
\end{array}\]}

\needspace{8\baselineskip}
For $k=15$, the point $2T$ gives the solution with $\digitsTwoTKFifteen$ digits in $e$:
{\scriptsize\setlength{\arraycolsep}{4pt}\[\begin{array}{rrrrr}
a & b & c & d & e \\ \hline
\begin{array}[t]{@{}r@{}}-406{,}215{,}715{,}303\\360{,}866{,}967{,}790\\156{,}660{,}951{,}289\\435{,}705{,}967{,}293\\909{,}448{,}416{,}059\\659{,}404{,}905{,}053\\423{,}089{,}906{,}024\\051{,}559\end{array} & \begin{array}[t]{@{}r@{}}-349{,}534{,}813{,}098\\668{,}928{,}660{,}624\\307{,}589{,}248{,}234\\899{,}705{,}925{,}991\\398{,}004{,}645{,}011\\584{,}026{,}897{,}166\\608{,}902{,}719{,}218\\860{,}050\end{array} & \begin{array}[t]{@{}r@{}}766{,}733{,}259{,}344\\623{,}303{,}115{,}080\\270{,}501{,}699{,}320\\876{,}970{,}015{,}776\\863{,}082{,}505{,}502\\223{,}757{,}164{,}829\\551{,}930{,}575{,}575\\778{,}100\end{array} & \begin{array}[t]{@{}r@{}}-416{,}839{,}850{,}214\\536{,}080{,}956{,}647\\427{,}903{,}637{,}754\\438{,}864{,}473{,}560\\325{,}413{,}543{,}064\\452{,}520{,}402{,}968\\996{,}134{,}384{,}408\\596{,}740\end{array} & \begin{array}[t]{@{}r@{}}804{,}045{,}890{,}733\\379{,}688{,}136{,}015\\498{,}084{,}042{,}652\\662{,}998{,}792{,}552\\457{,}622{,}897{,}322\\173{,}888{,}437{,}017\\342{,}377{,}919{,}560\\369{,}691\end{array} \\
\end{array}\]}

\needspace{8\baselineskip}
For $k=1/2$, the point $2T$ gives the solution with $\digitsTwoTKHalf$ digits in $e$:
{\scriptsize\setlength{\arraycolsep}{4pt}\[\begin{array}{rrrrr}
a & b & c & d & e \\ \hline
\begin{array}[t]{@{}r@{}}574{,}484{,}820{,}373\\482{,}697{,}542{,}553\\597{,}022{,}056{,}354\\626{,}042{,}607{,}075\\239{,}205{,}408{,}546\\442{,}676{,}943{,}636\\052{,}000{,}285{,}638\\985{,}770{,}682{,}545\\909\end{array} & \begin{array}[t]{@{}r@{}}-98{,}238{,}004{,}147\\226{,}439{,}752{,}171\\740{,}530{,}495{,}465\\120{,}992{,}747{,}884\\463{,}212{,}481{,}835\\768{,}959{,}839{,}013\\117{,}145{,}222{,}028\\067{,}884{,}313{,}227\\440\end{array} & \begin{array}[t]{@{}r@{}}-42{,}579{,}585{,}202\\836{,}707{,}541{,}451\\748{,}814{,}020{,}806\\984{,}040{,}117{,}660\\540{,}380{,}171{,}546\\376{,}228{,}559{,}450\\230{,}401{,}878{,}323\\952{,}551{,}355{,}950\\740\end{array} & \begin{array}[t]{@{}r@{}}406{,}847{,}605{,}015\\543{,}872{,}041{,}167\\038{,}708{,}125{,}895\\229{,}711{,}756{,}741\\990{,}776{,}259{,}345\\920{,}848{,}108{,}413\\838{,}526{,}841{,}912\\763{,}041{,}196{,}168\\180\end{array} & \begin{array}[t]{@{}r@{}}607{,}738{,}572{,}331\\667{,}788{,}135{,}996\\540{,}692{,}507{,}557\\516{,}627{,}468{,}474\\862{,}603{,}359{,}291\\914{,}634{,}407{,}379\\863{,}372{,}753{,}334\\078{,}596{,}373{,}419\\659\end{array} \\
\end{array}\]}

\needspace{8\baselineskip}
For $k=1/5$, the point $2T$ gives the solution with $\digitsTwoTKFifth$ digits in $e$:
{\scriptsize\setlength{\arraycolsep}{4pt}\[\begin{array}{rrrrr}
a & b & c & d & e \\ \hline
\begin{array}[t]{@{}r@{}}6{,}147{,}298{,}291\\596{,}488{,}640{,}633\\600{,}211{,}308{,}101\\969{,}880{,}076{,}192\\518{,}250{,}428{,}053\\134{,}267{,}647{,}809\\580{,}583{,}880{,}210\\215{,}436{,}722{,}847\\746\end{array} & \begin{array}[t]{@{}r@{}}536{,}173{,}172{,}943\\200{,}040{,}889{,}611\\974{,}172{,}118{,}675\\114{,}705{,}909{,}497\\472{,}264{,}455{,}584\\870{,}081{,}242{,}886\\730{,}540{,}541{,}772\\481{,}291{,}552{,}815\end{array} & \begin{array}[t]{@{}r@{}}37{,}804{,}197{,}259\\826{,}189{,}056{,}335\\343{,}874{,}492{,}968\\001{,}026{,}562{,}975\\539{,}212{,}731{,}671\\019{,}066{,}516{,}411\\985{,}317{,}967{,}551\\524{,}221{,}523{,}370\end{array} & \begin{array}[t]{@{}r@{}}-550{,}505{,}386{,}220\\467{,}506{,}183{,}522\\189{,}453{,}024{,}566\\145{,}608{,}103{,}564\\400{,}380{,}443{,}150\\670{,}302{,}393{,}193\\068{,}945{,}771{,}259\\316{,}459{,}619{,}810\end{array} & \begin{array}[t]{@{}r@{}}6{,}147{,}486{,}067\\468{,}349{,}110{,}423\\699{,}612{,}336{,}850\\666{,}495{,}837{,}187\\469{,}519{,}316{,}827\\087{,}109{,}398{,}572\\509{,}429{,}055{,}512\\119{,}954{,}235{,}275\\397\end{array} \\
\end{array}\]}

\needspace{8\baselineskip}
For $k=5/3$, the point $2T$ gives the solution with $\digitsTwoTKFiveThirds$ digits in $e$:
{\scriptsize\setlength{\arraycolsep}{4pt}\[\begin{array}{rrrrr}
a & b & c & d & e \\ \hline
\begin{array}[t]{@{}r@{}}-41{,}172{,}007{,}377\\156{,}946{,}833{,}944\\944{,}896{,}645{,}788\\505{,}710{,}887{,}560\\285{,}237{,}064{,}507\\748{,}998{,}528{,}368\\498{,}243{,}739\end{array} & \begin{array}[t]{@{}r@{}}142{,}032{,}516{,}691\\422{,}541{,}209{,}970\\403{,}397{,}658{,}031\\098{,}118{,}327{,}371\\790{,}172{,}680{,}525\\281{,}031{,}296{,}194\\422{,}129{,}300\end{array} & \begin{array}[t]{@{}r@{}}-98{,}100{,}912{,}282\\931{,}065{,}040{,}349\\486{,}509{,}924{,}636\\080{,}388{,}120{,}546\\746{,}243{,}727{,}926\\417{,}887{,}070{,}004\\531{,}569{,}250\end{array} & \begin{array}[t]{@{}r@{}}-2{,}699{,}418{,}040\\896{,}048{,}275{,}945\\178{,}884{,}830{,}787\\763{,}403{,}958{,}004\\220{,}243{,}209{,}340\\343{,}554{,}143{,}543\\166{,}280{,}860\end{array} & \begin{array}[t]{@{}r@{}}149{,}717{,}744{,}324\\673{,}737{,}858{,}998\\286{,}598{,}273{,}689\\523{,}577{,}301{,}425\\009{,}604{,}339{,}466\\878{,}733{,}335{,}736\\706{,}752{,}511\end{array} \\
\end{array}\]}

\needspace{8\baselineskip}
For $k=5/4$, the point $2T$ gives the solution with $\digitsTwoTKFiveQuarters$ digits in $e$:
{\scriptsize\setlength{\arraycolsep}{4pt}\[\begin{array}{rrrrr}
a & b & c & d & e \\ \hline
\begin{array}[t]{@{}r@{}}-18{,}656{,}694{,}544\\697{,}661{,}144{,}334\\209{,}508{,}023{,}971\\515{,}330{,}425{,}527\\379{,}625{,}195{,}763\\062{,}443{,}432{,}898\\981{,}275{,}280{,}463\\449{,}015{,}017\end{array} & \begin{array}[t]{@{}r@{}}598{,}693{,}660{,}870\\825{,}063{,}160{,}896\\418{,}186{,}393{,}243\\547{,}272{,}301{,}313\\489{,}997{,}044{,}621\\670{,}774{,}654{,}081\\956{,}424{,}776{,}818\\581{,}257{,}510\end{array} & \begin{array}[t]{@{}r@{}}636{,}319{,}412{,}285\\439{,}525{,}415{,}593\\560{,}969{,}147{,}344\\361{,}054{,}977{,}366\\369{,}344{,}759{,}850\\070{,}823{,}242{,}396\\121{,}945{,}857{,}770\\717{,}558{,}170\end{array} & \begin{array}[t]{@{}r@{}}-767{,}590{,}651{,}647\\580{,}492{,}725{,}076\\172{,}229{,}564{,}743\\689{,}640{,}449{,}260\\710{,}709{,}008{,}388\\734{,}075{,}330{,}313\\883{,}147{,}741{,}077\\817{,}464{,}000\end{array} & \begin{array}[t]{@{}r@{}}894{,}277{,}722{,}464\\450{,}963{,}565{,}458\\382{,}144{,}272{,}599\\224{,}292{,}288{,}181\\895{,}048{,}234{,}086\\561{,}624{,}079{,}140\\462{,}519{,}433{,}426\\162{,}999{,}983\end{array} \\
\end{array}\]}

\needspace{8\baselineskip}
For $k=11/7$, the point $2T$ gives the solution with $\digitsTwoTKElevenSevenths$ digits in $e$:
{\scriptsize\setlength{\arraycolsep}{4pt}\[\begin{array}{rrrrr}
a & b & c & d & e \\ \hline
\begin{array}[t]{@{}r@{}}-1{,}417{,}798{,}786\\929{,}386{,}472{,}850\\417{,}731{,}827{,}526\\326{,}255{,}900{,}087\\231{,}452{,}499{,}491\\843{,}428{,}443{,}899\\531{,}666{,}168{,}192\\510{,}405{,}867{,}577\\642{,}754{,}340{,}083\\865{,}200{,}026{,}177\\246{,}581{,}641{,}323\\715{,}912{,}474{,}514\\841{,}110\end{array} & \begin{array}[t]{@{}r@{}}3{,}367{,}149{,}058\\080{,}145{,}648{,}410\\274{,}235{,}818{,}170\\226{,}046{,}244{,}961\\441{,}643{,}676{,}475\\338{,}398{,}100{,}319\\229{,}130{,}927{,}472\\060{,}635{,}242{,}464\\133{,}338{,}974{,}370\\536{,}047{,}315{,}464\\721{,}318{,}956{,}655\\326{,}324{,}807{,}180\\453{,}765\end{array} & \begin{array}[t]{@{}r@{}}12{,}493{,}409{,}757\\599{,}869{,}021{,}820\\301{,}470{,}175{,}504\\257{,}367{,}191{,}281\\464{,}416{,}663{,}625\\249{,}321{,}065{,}611\\680{,}586{,}736{,}758\\679{,}548{,}439{,}222\\420{,}220{,}529{,}131\\344{,}717{,}556{,}576\\559{,}762{,}713{,}714\\608{,}301{,}822{,}445\\825{,}474\end{array} & \begin{array}[t]{@{}r@{}}-11{,}769{,}589{,}325\\195{,}932{,}672{,}835\\376{,}834{,}864{,}371\\195{,}610{,}392{,}949\\556{,}081{,}512{,}954\\132{,}793{,}756{,}786\\155{,}155{,}027{,}156\\916{,}234{,}215{,}199\\993{,}256{,}511{,}029\\490{,}774{,}026{,}021\\487{,}990{,}237{,}068\\618{,}178{,}589{,}207\\806{,}430\end{array} & \begin{array}[t]{@{}r@{}}14{,}457{,}071{,}160\\109{,}427{,}342{,}114\\625{,}117{,}948{,}283\\224{,}956{,}492{,}400\\957{,}532{,}395{,}353\\437{,}929{,}339{,}108\\247{,}700{,}799{,}577\\925{,}969{,}467{,}389\\155{,}389{,}342{,}075\\758{,}933{,}548{,}319\\384{,}918{,}935{,}131\\246{,}747{,}414{,}106\\113{,}143\end{array} \\
\end{array}\]}

\needspace{8\baselineskip}
For $k=25/13$, the point $2T$ gives the solution with $\digitsTwoTKTwentyfiveThirteenths$ digits in $e$:
{\scriptsize\setlength{\arraycolsep}{4pt}\[\begin{array}{rrrrr}
a & b & c & d & e \\ \hline
\begin{array}[t]{@{}r@{}}-334{,}148{,}341{,}439\\654{,}108{,}664{,}664\\433{,}754{,}833{,}569\\778{,}457{,}088{,}307\\809{,}243{,}986{,}899\\214{,}773{,}073{,}145\\753{,}732{,}312{,}412\\856{,}160{,}386{,}001\\434{,}321{,}720{,}545\\687\end{array} & \begin{array}[t]{@{}r@{}}225{,}622{,}948{,}634\\646{,}229{,}287{,}417\\999{,}068{,}976{,}270\\393{,}199{,}375{,}907\\840{,}255{,}975{,}562\\514{,}031{,}898{,}665\\919{,}506{,}709{,}139\\083{,}582{,}483{,}887\\002{,}899{,}263{,}947\\050\end{array} & \begin{array}[t]{@{}r@{}}296{,}627{,}800{,}187\\605{,}209{,}241{,}128\\122{,}903{,}948{,}346\\902{,}898{,}792{,}338\\874{,}383{,}051{,}045\\140{,}610{,}073{,}927\\999{,}781{,}850{,}968\\926{,}216{,}128{,}121\\852{,}732{,}863{,}481\\850\end{array} & \begin{array}[t]{@{}r@{}}-408{,}558{,}275{,}195\\052{,}041{,}125{,}410\\709{,}098{,}421{,}737\\006{,}073{,}002{,}707\\062{,}093{,}809{,}152\\412{,}686{,}341{,}438\\482{,}246{,}180{,}668\\310{,}916{,}616{,}621\\797{,}285{,}326{,}592\\850\end{array} & \begin{array}[t]{@{}r@{}}474{,}429{,}219{,}063\\300{,}185{,}565{,}645\\455{,}259{,}325{,}512\\848{,}599{,}448{,}586\\624{,}537{,}998{,}871\\452{,}298{,}723{,}305\\636{,}248{,}196{,}802\\116{,}543{,}836{,}994\\827{,}248{,}039{,}610\\563\end{array} \\
\end{array}\]}

\needspace{8\baselineskip}
For $k=20/29$, the point $2T$ gives the solution with $\digitsTwoTKTwentyTwentyninths$ digits in $e$:
{\scriptsize\setlength{\arraycolsep}{4pt}\[\begin{array}{rrrrr}
a & b & c & d & e \\ \hline
\begin{array}[t]{@{}r@{}}10{,}247{,}903{,}337\\734{,}773{,}013{,}296\\618{,}874{,}913{,}990\\046{,}851{,}987{,}284\\489{,}510{,}721{,}827\\589{,}842{,}646{,}032\\016{,}832{,}447{,}782\\251{,}723{,}618{,}627\\258{,}593{,}450{,}943\\030{,}260{,}773{,}506\\001\end{array} & \begin{array}[t]{@{}r@{}}32{,}569{,}217{,}648\\469{,}519{,}789{,}371\\138{,}883{,}081{,}040\\485{,}344{,}387{,}594\\687{,}997{,}098{,}760\\251{,}392{,}975{,}490\\104{,}359{,}336{,}942\\656{,}484{,}304{,}733\\436{,}461{,}126{,}505\\183{,}754{,}019{,}312\\280\end{array} & \begin{array}[t]{@{}r@{}}27{,}346{,}452{,}412\\199{,}192{,}953{,}014\\871{,}330{,}456{,}325\\701{,}112{,}214{,}893\\794{,}127{,}864{,}621\\248{,}832{,}765{,}117\\586{,}696{,}072{,}373\\231{,}654{,}173{,}773\\627{,}778{,}775{,}264\\324{,}622{,}098{,}611\\320\end{array} & \begin{array}[t]{@{}r@{}}23{,}482{,}594{,}235\\762{,}190{,}121{,}811\\681{,}473{,}178{,}244\\905{,}031{,}991{,}204\\373{,}943{,}907{,}027\\409{,}674{,}199{,}605\\070{,}224{,}783{,}434\\727{,}905{,}873{,}372\\005{,}123{,}197{,}564\\539{,}432{,}576{,}249\\360\end{array} & \begin{array}[t]{@{}r@{}}37{,}603{,}929{,}184\\282{,}283{,}075{,}766\\647{,}719{,}627{,}790\\888{,}703{,}180{,}426\\596{,}548{,}770{,}262\\182{,}454{,}049{,}603\\384{,}672{,}110{,}827\\269{,}082{,}051{,}323\\660{,}861{,}063{,}172\\821{,}661{,}386{,}626\\001\end{array} \\
\end{array}\]}